\documentclass{amsart}

\usepackage{amsmath}
\usepackage{url}
\usepackage{amsthm}
\usepackage{amsfonts}
\usepackage{hyperref}

\newtheorem{theorem}{Theorem}[section]
\newtheorem{lemma}[theorem]{Lemma}
\newtheorem{proposition}[theorem]{Proposition}
\newtheorem{definition}[theorem]{Definition}
\newtheorem{corollary}[theorem]{Corollary}

\newtheorem{prop}{Proposition}[section]
\newtheorem{remark}[prop]{Remark}

\makeatletter \@addtoreset{equation}{section} \makeatother

\newcommand{\h}{\textbf{h}}
\newcommand{\X}{\textbf{X}}
\newcommand{\dd}{\cdot}

\begin{document}

\title[]{Index Characterization for Free Boundary Minimal Surfaces}
\author{Hung Tran}
\address{Department of Mathematics and Statistics,
 Texas Tech University, Lubbock, TX 79409}
\email{hung.tran@ttu.edu}

\renewcommand{\subjclassname}{%
  \textup{2000} Mathematics Subject Classification}
\date{}

\begin{abstract} In this paper, we compute the Morse index of a free boundary minimal submanifold from data of two simpler problems. The first is the fixed boundary problem and the second is concered with the Dirichlet-to-Neumann map associated with the Jacobi operator. As an application, we show that the Morse index of a free boundary minimal annulus is equal to 4 if and only if it is the critical catenoid.   

\end{abstract}
\maketitle
\tableofcontents

\section{Introduction}
The goal of this paper is to study the Morse index of a free boundary minimal submanifold (FBMS), particularly of codimension one. Given an orientable manifold $\Omega^n$ with boundary $\partial \Omega$, a FBMS is a critical point of the volume functional among all submanifolds with boundaries in $\partial \Omega$. 
As a consequence, a properly immersed $\Sigma\subset \Omega$ is a FBMS if and only if its mean curvature vanishes and $\partial M$ meets $\partial \Omega$ perpendicularly. The simplest example is an equatorial plane in the unit Euclidean ball. Another simple known example is the critical catenoid with rotational symmetry (see Section \ref{fbma}).  

Due to that intriguingly geometrical combination of minimality and boundary orthogonality, the subject has attracted widespread interest which can be traced back to \cite{gergonne1816, schwarz1890}. Several existence results were obtained by different methods including various minimax constructions, varifold theory, minimizing the Dirichlet functional and maximizing the first Steklov eigenvalue; see \cite{courant40, MY80, smyth84, Struwe84, Ye_free91, GJ86, Jost91, Fraser00, Li_free15, MNS16, FS16, FGM16} and references therein. There has also been extensive research aimed at understanding the boundary regularity of FBMS (see \cite{lewy51,jager70, HN81, GJ86regularity} and the excellent survey in \cite[Chapter 2]{DHTKv2}). Roughly speaking, when $\partial \Omega$ is smooth enough, the boundary of a FBMS is as smooth as $\partial \Omega$. In particular, if $\partial \Omega$ is real analytic then $\Sigma$ is real analytic and can be continued analytically across the boundary. \\ 

Recently, the subject has gained even more popularity due to a new perspective related to extremal metrics for Steklov eigenvalues, primarily due to the work of A. Fraser and R. Schoen \cite{FS11, FS15, FS16}. As this concept is important to our approach, let us explain it.  

Steklov eigenvalues are associated with the harmonic extension of functions defined on the boundary. Specifically, given $h\in C^{\infty}(\partial \Sigma)$, we consider the problem:
\[
\Bigg\{ 
\begin{tabular}{cc}
$\Delta \hat{h}=0$ & \text{on $\Sigma$},\\
$\hat{h}=h$ & \text{on $\partial \Sigma$}.
\end{tabular}
\]
The Dirichlet-to-Neumann map associated with the Laplacian, 
\[ L_{\Delta}: C^{\infty}(\partial M) \mapsto C^{\infty}(\partial M),\]
is given by 
\begin{equation}\label{LapSteklov}
 L_{\Delta}h=\frac{\partial \hat{h}}{\partial \eta}.\end{equation}
It is well known that $\Delta$ is an elliptic self-adjoint operator and that the harmonic extension has a unique solution. As a consequence, $L_{\Delta}$ is a non-negative self-adjoint operator with discrete spectrum $0=\xi_0<\xi_1\leq \xi_2\leq ...$ tending to infinity. The elements of the spectrum are called Steklov eigenvalues. 
 
In connection with our earlier discussion, when $\Omega$ is a Euclidean ball, the coordinate functions of a FBMS are eigenfunctions with Steklov eigenvalue 1. On the hand, it is well-known that the coordinate functions of a minimal submanifold in a sphere are eigenfunctions of the Laplacian. These observations provide an analogy between two settings. 

Our main focus here is to better understand the Morse index, which intuitively gives the number of distinct admissible deformations which decrease the volume to second order. One motivation is from the analogy to minimal submanifolds in a sphere, where remarkable results have been obtained recently. For minimal surfaces in $\mathbb{S}^3$, due to J. Simons \cite{Simons68}, the index is at least 1 and equality happens only for the totally geodesic immersion. Then a non-totally geodesic minimal surface has index at least 5 and, due to F. Urbano \cite{urbano90}, the Clifford torus is the only one with that index. What is more, that index characterization plays a key role in the recent celebrated proof of the Willmore conjecture by F. Marques and A. Neves \cite{MN_minmax14}. 

For a FBMS, less is known. There are restrictions on the topology of a FBMS with low index under some curvature assumptions \cite{ros08, CFP15}. Also, we mention recent papers giving lower estimates of the index by topological data \cite{Sargent16, ACS16}. It is likely that those inequalities are not sharp.

If $\Omega$ is a Euclidean ball, then the equatorial disk has index 1 (see \cite{Fraser07} or Remark \ref{equatorindex1}). 
It is conjectured that the critical catenoid is the only FBMS in $\mathbb{B}^3$ with index 4. In this direction, Fraser and Schoen showed that, if $\Sigma^k\subset \mathbb{B}^n$ is not a plane disk, then its index is at least $n$ \cite[Theorem 3.1]{FS16}. \\

This paper takes the following approach. Inspired by the work of Fraser and Schoen \cite{FS16}, we'll reduce the analysis of the Morse index into simpler point-wise problems,  studying variations fixing the boundary and Steklov eigenvalues associated with the Jacobi operator. As an application, we partially address the conjecture above by showing that the critical catenoid is the only free boundary minimal annulus with index 4. \\   

To describe our results, let us restrict our attention to when $\Sigma^k\subset \Omega^{k+1}$ is a smooth, properly immersed, and orientable FBMS. Thus $\Sigma$ is two-sided (when $\Sigma$ is one-sided one can consider its double cover). Consequently, there exists a smooth unit normal vector field $\nu$ and we may restrict our attention to normal variations of the form $V=u\nu$ for any smooth function $u$. The second variation of volume of $\Sigma(t)$ at $\Sigma=\Sigma(0)$ is the index bilinear form (\cite{MNS16})
\begin{align*}
{S}(u,u) &=\frac{d^2}{dt^2}\text{Vol}(\Sigma(t))\mid_{t=0}\\
 &=\int_{\Sigma}|\nabla^\Sigma u|^2-(\text{Rc}^{\Omega}(\nu,\nu)+|\h|^2)u^2 d\mu+\int_{\partial \Sigma}\left\langle{\nabla^{\Omega}_\nu \nu, \eta}\right\rangle u^2 da.
\end{align*}   
Here the superscripts indicate the context of corresponding operators; $|\h|$ is the norm of the second fundamental form of $\Sigma\subset \Omega$, $\text{Rc}$ denotes the Ricci tensor and $\eta$ is the outward conormal vector along $\partial \Sigma$ and is perpendicular to $\partial \Omega$. Therefore, \[\left\langle{\nabla^{\Omega}_\nu \nu,\eta}\right\rangle=\h^{\partial \Omega}(\nu,\nu),\]
where $\h^{\partial \Omega}(.,.)$ is the second fundamental form with respect to the outward unit normal of $\partial \Omega \subset \Omega$. 
\begin{definition}
The Morse index of $\Sigma^k\subset \Omega^{k+1}$ is the maximal dimension of a subspace of $C^\infty(\Sigma)$ on which the second variation is negative definite. The nullity is the dimension of the kernel of the index form; that is, the set of all $u$ such that $S(u,v)=0$ for all $v$.  
\end{definition}
Recall the Jacobi operator 
\begin{equation}\label{jacobi}
{J}=\Delta_\Sigma +\text{Rc}(\nu,\nu)+|\h^\Sigma|^2.
\end{equation}
It is well known (\cite{ACS16}, \cite{Sargent16}, \cite{MNS16}) that the index is equal to the number of negative eigenvalues counting multiplicity for the following system, 
\[ 
\Bigg\{ 
\begin{tabular}{cc}
$(\Delta_\Sigma +\text{Rc}(\nu,\nu)+|\h^\Sigma|^2) u={J}u=-\lambda u $ & \text{on $\Sigma$},\\
$\frac{\partial u}{\partial \eta}=-\h^{\partial \Omega}(\nu,\nu)u$ & \text{on $\partial \Sigma$}.
\end{tabular}\\
\]

If we restrict to variations fixing the boundary then the boundary integral disappears. That leads to the following system: 
\begin{equation}
\label{FBP}
\begin{cases}
{J}u &=-\lambda u \text{  on  } \Sigma,\\
u & \equiv 0 \text{  on  }\partial \Sigma.
\end{cases}
\end{equation}

The number of negative eigenvalues for (\ref{FBP}) is generally smaller than the Morse index because of the boundary condition. The influence of the boundary is then analyzed by the following Dirichlet-to-Neumann map associated with the Jacobi operator.  \\

Given a function $h\in C^{\infty} (\partial \Sigma)$, consider the Jacobi extension of $h$ (See Lemma \ref{existenceJ}):
\[
\Bigg\{ 
\begin{tabular}{cc}
${J}\hat{h} =0$ & \text{ on $\Sigma$},\\
$\hat{h}=h$ & \text{on $\partial M$}.
\end{tabular}
\]
Associated is its Dirichlet-to-Neumann map (see Subsection \ref{Dtn})
\begin{equation}\label{jacobiSteklov}
 L_{{J}} h=\frac{\partial \hat{h}}{\partial \eta}.\end{equation}
It turns out that $L_{{J}}$ has a discrete spectrum tending to infinity. 

Our first result characterizes the Morse index and nullity by data from the corresponding problem with fixed boundary and the Dirichlet-to-Neumann map associated with Jacobi operator. 

\begin{theorem}\label{main1} Let $\Sigma^k\subset \Omega^{k+1}$ be a smooth, properly immersed, orientable FBMS such that $\h^{\partial \Omega}(\nu,\nu)=-c$, a constant. Then, we have:
\begin{itemize}
\item The Morse index is equal to the number of non-positive eigenvalues of the fixed boundary problem (\ref{FBP}) plus the number of eigenvalues less than c of the Dirchlet-to-Neumann map (\ref{jacobiSteklov}), counting multiplicities. 
\item The nullity is equal to the dimension of the eigenspace with eigenvalue $c$ of (\ref{jacobiSteklov}).   
\end{itemize}
\end{theorem}
\begin{remark} It is clear that ${S}(.,.)$ is negative on eigenfunctions with negative eigenvalues of (\ref{FBP}). Next, if $\hat{h}$ is a Jacobi extension of $h$ such that $L_{J}\hat{h}=\delta h$ for $\delta<c$, then 
\[{S}(\hat{h},\hat{h})=(\delta-c)\int_{\partial \Sigma}h^2<0.\]
The crucial non-triviality in the proof is that eigenfunctions with eigenvalue 0 of (\ref{FBP}) can be modified to give negative deformations, see Lemma \ref{PropertyKB}.
\end{remark}
As an application, we give a partial index characterization of the critical catenoid. 
\begin{theorem}\label{main2}
The critical catenoid in $\mathbb{B}^3$ has Morse index 4 and nullity 2. Conversely, a free boundary minimal annulus in $\mathbb{B}^3$ with index 4 must be the critical catenoid. 
\end{theorem}
\begin{remark}
We learned the converse statement from Richard Schoen's lecture in 2015. Upon completion of this paper, it comes to our attention that catenoid having index 4 is independently proved by B. Devyver \cite{devyver16index} and G. Smith and D. Zhou \cite{SZ16index} by different methods. Also some results of Subsection \ref{index4} are independently 
observed by B. Devyver and A. Fraser \cite{devyver16index}.  
\end{remark}
The organization of the paper is as follows. Section \ref{notation} collects notation and preliminaries. Then, in Section \ref{indextheorem}, we give a proof of Theorem \ref{main1} and discuss a generalization and estimates in the case $\h^{\partial \Sigma}(\nu,\nu)$ is not constant. Finally, we compute the index of the critical catenoid in Section \ref{fbma}.  \\
{\bf Acknowledgments:} The author would like to thank Richard Schoen for inspiring lectures and extensive discussion. The author has also benefited greatly from conversations with Xiaodong Cao, David Wiygul, and Peter McGrath. Finally, the author is grateful to an anonymous referee for detailed and constructive suggestions. 
\section{Notation and Preliminaries}\label{notation}
This section collects notation and preliminary results.  

We adopt the following setting:
\begin{itemize}
\item $\Sigma^k$ is a smooth, orientable, properly immersed FBMS with boundary $\partial \Sigma$ in the smooth orientable manifold $\Omega^{k+1}$ with boundary $\partial \Omega$. We note that if $\Omega$ is simply connected then the orientability of $\Sigma$ is automatic.
\item When $\Omega=\mathbb{B}^{k+1}$, the Euclidean ball of radius 1, we let $X$ denote the position vector. 
\item $\eta$ denotes the outward conormal vector along the boundary. Note that if $\Sigma$ is a FBMS in $\mathbb{B}^{k+1}$ then $\eta=X$ on $\partial \Sigma$.
\item $\nu$ is a choice of normal vector to the surface such that, if $\Sigma$ is non-equatorial FBMS in $\mathbb{B}^{k+1}$, then 
$\zeta=\left\langle{X, \nu}\right\rangle$ is positive at some point. 
\item For $\Sigma^k \subset \Omega^{k+1}$, with respect to a local orthonormal frame $e_1,... e_k$ tangent $\Sigma$, the second fundamental form is defined as, 
\[\h^{\Sigma}_{ij}=\left\langle{\nabla^{\Omega}_{e_i}e_j, \nu}\right\rangle=-\left\langle{\nabla^{\Omega}_{e_i}\nu, e_j}\right\rangle.\]
 Then $|\h^{\Sigma}|$ denotes its norm and the mean curvature of $\Sigma$ is just its trace.  
\item For a fixed vector $a\in \mathbb{R}^n$, $X_a=\left\langle{X,a}\right\rangle,$ $\nu_a=\left\langle{\nu,a}\right\rangle.$
\item We'll drop the volume form when the context is clear. 
\end{itemize}

Recall the bilinear form associated with the second variation for a FBMS $\Sigma^k\subset \Omega^{k+1}$, for $h, f\in C^{\infty}(\Sigma)$,
\begin{align}
\label{2ndvarAREA}
{S}(f,h) &= \int_\Sigma \nabla f\nabla h-(\text{Rc}(\nu,\nu)+|\h^\Sigma|^2) fh+\int_{\partial \Sigma}\h^{\partial \Omega}(\nu,\nu) fh.
\end{align}

This motivates the following generalization. 

\begin{definition}
\label{generalsetup}
Given $\phi, m\in C^{\infty}(\Sigma)$ such that $\phi\geq \epsilon>0$ and a constant $\alpha$, we define:
\begin{align*}
J &=\nabla (\phi \nabla)+m;\\
Q(h,f) &=\int_\Sigma \phi \nabla f\nabla h-m fh,\\
S(h,f) &=\int_\Sigma \phi \nabla f\nabla h-mfh-\alpha \int_{\partial \Sigma}\phi fh\nonumber \\
&=Q(h,f)-\alpha\int_{\partial \Sigma}\phi fh.
\end{align*}
\end{definition}
Note that the choice of $J$ is such that,
\begin{align*}
\int_{\Sigma}hJf-fJh &=-\int_{\partial \Sigma} \phi (fD_{\eta}h-hD_{\eta}f),\\
S(u,v) &= \int_\Sigma \phi \nabla f\nabla h-mfh-\alpha \int_{\partial \Sigma}\phi fh\\
 &= -\int_\Sigma u Jv+\int_{\partial \Sigma} (D_\eta v-\alpha v) u\phi. 
\end{align*}
Consequently, we define the index and nullity associated with $S(\dd,\dd)$ as follows.
\begin{definition}
The index with respect to the bilinear form $S(\dd,\dd)$ is the maximal dimension of a subspace of $C^\infty(\Sigma)$ in which $S(\dd,\dd)$ is negative definite. The nullity is the dimension of the set of all $u$ such that $S(u,v)=0$ for all $v$. 
\end{definition}
\begin{remark}
When $S(\dd,\dd)$ is the bilinear form associated with the second variation formula, then the definition above recovers the Morse index of $\Sigma$. 
\end{remark}
\subsection{The Fixed Boundary Problem}
If we restrict to variations fixing the boundary then the boundary integral in (\ref{2ndvarAREA}) disappears. Similarly, to understand the index of $S(\dd,\dd)$, we first consider functions vanishing at the boundary. That will be relevant when we study the Dirichlet boundary value problem for Jacobi operator.\\

As $\phi\geq \epsilon >0$, it follows that $J$ is an elliptic and self-adjoint differential operator with compact resolvent. 
In particular, by spectral theory, $J$ has a discrete spectrum that goes to infinity. Each eigenfunction $u$ with eigenvalue $\lambda$ satisfies the following system,
\begin{equation}
\label{FBP1}
\begin{cases}
{J}u &=-\lambda u \text{ on } \Sigma,\\
u &\equiv 0 \text{ on } \partial \Sigma.
\end{cases}
\end{equation}

The eigenvalues can be characterized by the min-max principle: 
\begin{equation}\label{varfixedbdry}
 \lambda_k (J) =\min_{V_k\subset \mathcal{W}_0^1}\max_{u\in V_k} \frac{Q(u,u)}{\int_\Sigma u^2}, \end{equation}
where each $V_k$ is $k$-dimensional subspace of $\mathcal{W}_0^{1,2}(\Sigma)$.
\begin{definition}\label{fixed} Let $\mathcal{J}^{-}_0$($\mathcal{J}^0_0$) denote the space  of eigenfunctions with negative (zero) eigenvalues for (\ref{FBP1}). The dimension of $\mathcal{J}_0^0$ is the nullity of (\ref{FBP1}). 
\end{definition}

We note that each space above is finite dimensional. 
 As a consequence, the following will be crucial in later analysis. 

\begin{definition} \label{bdrydata}
When the nullity is positive, let $\{w_i, ~i=1,..,\dim(\mathcal{J}_0^0))\} $ be a basis of $\mathcal{J}_0^0$. We define, 
\begin{equation*}
D_\eta \mathcal{J}_0^0=\text{span}(b_i=D_\eta w_i, ~i=1,..,\dim(\mathcal{J}_0^0)).
\end{equation*}
Let $\overline{D_\eta \mathcal{J}_0^0}\subset C^\infty(\Sigma)$ be the space of all $u\in C^\infty(\Sigma)$ such that $u_{\mid\partial \Sigma}\in D_\eta \mathcal{J}_0^0$.
\end{definition}
\begin{remark}It is clear that $\dim(D_\eta \mathcal{J}_0^0)=\dim(\mathcal{J}_0^0)$.
\end{remark}
\begin{remark} For a FBMS in a Euclidean ball, $\dim(\mathcal{J}_0^0)>0$ (Subsection \ref{FBMSinball}).\end{remark}

\subsection{The Dirichlet-to-Neumann map} \label{Dtn}
Given a function $h\in C^{\infty} (\partial \Sigma)$, consider the $J$-extension (Dirichlet problem associated with operator $J$) of $h$:
\[\Bigg\{ 
\begin{tabular}{cc}
${J}\hat{h}=0 $ & \text{on $\Sigma$},\\
$\hat{h} = h$ & \text{on $\partial \Sigma$}.
\end{tabular}
\]

The following result is well known. 
\begin{lemma}\label{existenceJ}
Given $h\in C^{\infty}(\partial \Sigma)$, the $J$-extension exists and is unique up to an addition of $w\in \mathcal{J}_0^0$ if and only if, for all $b\in D_\eta \mathcal{J}_0^0$,  
\[\int_{\partial\Sigma}\phi b h=0.\]
In other words, $h\in (D_\eta\mathcal{J}_0^0)^{\perp}\subset C^{\infty}(\partial \Sigma)$ with respect to the $\phi$-weighted $L^2$ inner product over $\partial \Sigma$.
\end{lemma}
\begin{proof}
The kernel of $J$ with Dirichlet boundary data is $\mathcal{J}_0^0$. By the Fredholm alternative, the extension exists if and only if, for any extension $\overline{h}$ of $h$ and any $w\in \mathcal{J}_0^0$,
\begin{align*}
0 &=\int_\Sigma w J\overline{h}\\
&=\int_\Sigma \overline{h}J w +\int_{\partial \Sigma}\phi(w D_\eta\overline{h}-h D_\eta w)\\
&=-\int_{\partial \Sigma}\phi h D_{\eta} w.
\end{align*}
Since $D_{\eta} w\in D_\eta\mathcal{J}_0^0$ the result follows.
\end{proof} 

Even though the extension is generally not unique, it is unique up to an addition of $\mathcal{J}_0^0$. Consequently, the single-valued Dirichlet-to-Neumann operator on $(D_\eta\mathcal{J}_0^0)^{\perp}$ is defined as follows. For $h\in (D_\eta\mathcal{J}_0^0)^{\perp}$ there is a unique $J$-extension $\hat{h}$ such that $D_\eta \hat{h} \in (D_\eta\mathcal{J}_0^0)^{\perp}$. Then,
\[ L_J: (D_\eta\mathcal{J}_0^0)^{\perp} \mapsto (D_\eta\mathcal{J}_0^0)^{\perp}\]
is given by
\begin{equation}\label{singleDtN}
L_J h=D_\eta \hat{h}.
\end{equation}

 
\begin{remark} Since $J$ with Dirichlet boundary condition possibly has non-trivial kernel, it is possible to define a multi-valued Dirichlet-to-Neumann map. Due to the symmetry of the corresponding functional and compactness of the trace operator, the multi-valued operator is self-adjoint with compact resolvent and bounded below (see \cite[Thm 4.5, Prop 4.8, Thm 4.15]{AEKS14} and \cite{AM12} for details). 
In particular, the single-valued part defined as above is a self-adjoint operator with discrete spectrum. The elements in that spectrum are called $J$-Steklov eigenvalues. 
\end{remark}

The $J$-Steklov eigenvalues can be characterized variationally. Let $V_k\subset (D_\eta\mathcal{J}_0^0)^{\perp}$ denote a $k$-dimensional subspace, then
\begin{align}\label{QandS}
\delta_k (L_J) &=\min_{V_k\subset (D_\eta\mathcal{J}_0^0)^{\perp}}\max_{h\in V_k} \frac{Q(\hat{h},\hat{h})}{\int_{\partial\Sigma} \phi h^2},
\end{align}
where $\hat{h}$ is any $J$-extension of $h$. 


\begin{remark}
As a consequence, for $L_J \hat{h}=\delta h$ and $\delta<\alpha$, $S(\hat{h},\hat{h})<0$. Therefore, $J-$Steklov eigenvalues play a role in the analysis of the index. 
\end{remark}
In particular, we introduce the following notations. 
\begin{definition}\label{defineHpm} We let $E_\delta$ denote the eigenspace of $L_J$ associated with the eigenvalue $\delta$. Then,
\begin{align*}
\hat{E}_\delta &=\{\hat{h} \mid h\in E_\delta \text{ and } D_\eta \hat{h}=\delta h \}.
\end{align*}

\end{definition}  
\begin{remark} Note that elements of $\bigoplus_{\delta<\infty} \hat{E}_\delta$ that are $L^2$-orthogonal on $\partial \Sigma$ are also orthogonal with respect to the bilinear form $S(\dd, \dd)$.
\end{remark}  
\begin{remark}
Since $L_J$ is bounded below and the spectrum is discrete, $\bigoplus_{\delta<c}E_\delta$ is of finite dimension.
\end{remark}


\subsection{FBMS in a Euclidean ball}
\label{FBMSinball}
This subsection applies the abstract setting above to the concrete case of a FBMS in the unit Euclidean ball $\mathbb{B}^{k+1}$.   
\begin{align*}
\text{Rc}(\nu,\nu) &=0,\\
\h^{\partial \mathbb{B}^{k+1}}(\nu,\nu)&=-1.
\end{align*}
So, choosing $\phi\equiv 1=\alpha$, $m=|\h^\Sigma|^2$ yields:
\begin{align*}
J&=\Delta +|\h^\Sigma|^2.
\end{align*}
For simplicity, we denote $\h^\Sigma$ by $\h$. The minimality implies,
\begin{align*}
\Delta X_a &=0,\\
J \nu_a &=0.
\end{align*}
It follows that, for $\zeta=\left\langle{X, \nu}\right\rangle$, 
\[(\Delta +|A|^2)\zeta =0.
\]
Similarly, for a skew-symmetric matrix $M\in \mathfrak{so}(n)$, the Lie algebra of the rotation group $SO(n)$,  
\[(\Delta +|A|^2)\left\langle{MX,\nu}\right\rangle =0.
\]
Here, $\left\langle{MX,\nu}\right\rangle$ represents the normal speed corresponding to a rotation determined by $M$. In $\mathbb{R}^3$, using the cross product $\times$, the infinitesimal normal speed associated with a rotation around a constant vector $a$ is $\left\langle{X\times \nu, a}\right\rangle.$ 

Next, we consider behavior along the boundary. First, the perpendicular boundary condition implies $\zeta_{\mid \partial \Sigma}=0$ so $\zeta\in \mathcal{J}_0^0$. Second, for a FBMS, $X_a$ is an eigenfunction with eigenvalue 1 of (\ref{LapSteklov}) (see \cite{FS11, FS16} for more details). That is, 
\begin{align*}
D_\eta X_a &=X_a,\\
0 &=\int_{\partial \Sigma}X_a.
\end{align*}

Then, it is interesting to study boundary derivative. Towards that end, it is observed that the free boundary condition implies that, along $\partial \Sigma$, $\h$ is diagonalized by $X$ and tangential vectors to $\partial \Sigma$. As a consequence, we compute, 
\begin{align*}
D_{\eta} \nu_a &=D_X \left\langle{\nu, a}\right\rangle=-\h(X,X)X_a=-\h(\eta, \eta)X_a,\\ 
D_{\eta} \zeta &=D_X \left\langle{X, \nu}\right\rangle =-\h(\eta,\eta),\\
D_\eta \left\langle{MX,\nu}\right\rangle &=\left\langle{MX,\nu}\right\rangle.
\end{align*}
Therefore, along $\partial \Sigma$, $ \h(X,X)\in D_\eta\mathcal{J}_0^0$ and $\nu_a, \left\langle{MX,\nu}\right\rangle \in (D_\eta\mathcal{J}_0^0)^\perp$. That is, 
\[\int_{\partial \Sigma}\h(X,X)\nu_a=0=\int_{\partial \Sigma}\h(X,X)\left\langle{MX,\nu}\right\rangle.\]
Furthermore, $\left\langle{MX,\nu}\right\rangle$ is an eigenfunction of (\ref{jacobiSteklov}) with eigenvalue 1. 

The following lemma generalizes \cite[Prop. 3.1]{FS16}. The proof is analogous and provided for completeness.  
\begin{lemma}
\label{Scross} Let $\Sigma^k\subset \mathbb{B}^{k+1}$ be a properly immersed FBMS. We have, for $a, b$ constant unit vectors,
\begin{align*}
S(\nu_a, \nu_b)&=\int_{\partial \Sigma}kX_a X_b-\left\langle{a,b}\right\rangle,\\
S(\nu_a, \nu_a) &=-k\int_{\Sigma} \nu_a^2.\end{align*}
\end{lemma}
\begin{proof} We compute, for $J\nu_a=J\nu_b=0$ inside $\Sigma$,
\begin{align*}
S(\nu_a, \nu_b) &= \int_{\partial \Sigma}\nu_a D_\eta \nu_b-\nu_a\nu_b\\
&=-\int_{\partial \Sigma}\nu_a \h(X,X) X_b+\nu_a\nu_b.
\end{align*}

Along $\partial \Sigma$, let $\{e_i\}_{i=1}^{k-1}, e_k=X=\eta$ be a local orthogonal frame  and $(\cdot)^T$ the tangential component of a vector field. Then, $a=X_a X+\nu_a \nu+ a^T$ and
\begin{align*}
\text{div}_{\partial \Sigma} (\nu_a \nu+ a^T)&=\text{div}_{\partial \Sigma} (a-X_a X)\\
&=-(k-1)X_a.
\end{align*}
On the other hand,
\begin{align*}
\text{div}_{\partial \Sigma} (\nu_a \nu+ a^T)&=\text{div}_{\partial \Sigma}(a^T)+\sum_{i}\nu_a \left\langle{\nabla_{e_i} \nu, e_i}\right\rangle\\
 &=\text{div}_{\partial \Sigma}(a^T)-\sum_{i}\nu_a \h(e_i, e_i)=\text{div}_{\partial \Sigma}(a^T)+\nu_a \h(\eta, \eta). 
\end{align*}
Therefore, by applying the divergence theorem, 
\begin{align*}
\int_{\partial \Sigma}\nu_a D_\eta \nu_b &=\int_{\partial \Sigma} X_b(\text{div}_{\partial \Sigma}(a^T)+(k-1)X_a)\\
&=\int_{\partial \Sigma}(k-1)X_a X_b-\int_{\partial \Sigma}a^T b^T.
\end{align*}
Consequently, 
\begin{align*}
S(\nu_a, \nu_b) &=\int_{\partial \Sigma}\nu_a D_\eta \nu_b-\nu_a\nu_b\\
&=\int_{\partial \Sigma}(k-1)X_a X_b-\int_{\partial \Sigma}(\left\langle{a^T,b^T}\right\rangle+\nu_a\nu_b)=\int_{\partial \Sigma}kX_a X_b-\left\langle{a,b}\right\rangle.
\end{align*}

When $a=b$, we consider the vector field $V=X-kX_a a$. By divergence theorem again,
\begin{align*}
\int_\Sigma k \nu_a^2=\int_\Sigma \text{div}(V)&=\int_{\partial \Sigma} \left\langle{X,V}\right\rangle=\int_{\partial \Sigma} (1-k X_a^2)=-S(\nu_a, \nu_a).
\end{align*}
That concludes the proof. 
\end{proof}
A consequence is the following whose proof is also analogous to \cite[Prop 3.1]{FS16}. 

\begin{corollary}
\label{lowH1}
Let $\Sigma^k\subset \mathbb{B}^{k+1}$ be a properly immersed FBMS. If $\Sigma$ is not equatorial, then
\[\dim(\bigoplus_{\delta<1}E_\delta)\geq k+1.\]
\end{corollary}
\begin{proof} Let $\{e_i\}_{i=1}^{k+1}$ be an orthonormal basis of $\mathbb{R}^{k+1}$, $\nu_i=\left\langle{\nu, e_i}\right\rangle$ and  \[V=\text{span}(\nu_1, ..., \nu_{k+1}).\] By Lemma \ref{Scross}, $S(\dd,\dd)$ is negative definite on $V$. Furthermore, for any $w\in V$, $w\mid_{\partial \Sigma} \in (D_\eta \mathcal{J}_0^0)^\perp$. 
Also $\dim(V)=k+1$ (otherwise, there is a constant vector $a$ such that $\nu_a=0$, which implies that $\Sigma$ is equatorial). Finally, we observe that
\[{Q(\nu_a,\nu_a)}=S(\nu_a, \nu_a)+||h||_{L^2(\partial \Sigma)}^2<||h||_{L^2(\partial \Sigma)}^2.\]  
The result then follows from the min-max characterization (\ref{QandS}).  
\end{proof}

We also observe a preliminary estimate for the first eigenvalue of $L_J$. 


\begin{proposition} \label{lesseq0}Let $\delta_1$ be the first eigenvalue of $L_J$. Then
\[\delta_1\leq 0,\]
and equality occurs if $\Sigma$ is flat. 
\end{proposition}
\begin{proof}
We compute, as in Lemma \ref{Scross},
\begin{align*}
S(\nu_i, \nu_i) &= \int_{\partial \Sigma} k X_i^2-1,\\
\sum_{i=1}^{k+1} S(\nu_i, \nu_i) &= \int_{\partial \Sigma}(k-(k+1))=-L(\partial \Sigma).
\end{align*}
Each $\nu_i$ is a Jacobi field and can be used as a test function for the variational characterization (\ref{QandS}) of $\delta_1$. Therefore, 
\begin{align*}
S(\nu_i, \nu_i) &\geq (\lambda_1-1) \int_{\partial \Sigma} \nu_i^2;\\
\sum_{i=1}^{k+1} S(\nu_i, \nu_i) &\geq (\lambda_1-1)L(\partial \Sigma).
\end{align*}

Combining these equations yields that $\lambda_1\leq 0$.   

Now if $\Sigma$ is flat, then $|\h|^2=0$. Consequently, the Jacobi operator reduces to the regular Laplacian and so the result follows. In this case, $\nu$ is a constant vector. 
\end{proof}

There is a partial result in the reverse direction. We first recall the following which was noted in the proof of Prop 8.1 of \cite{FS16}. 
\begin{lemma} \label{fixedzero}
Let $\Sigma^k\subset \mathbb{B}^{k+1}$ be a properly immersed FBMS with $\zeta$ positive everywhere inside $\Sigma$. Then the first eigenvalue for the fixed boundary problem (\ref{FBP1}) is 0. In particular, $\mathcal{J}_0^-=\emptyset$ and $\mathcal{J}_0^{0}=\text{span}(\zeta)$.
\end{lemma}
\begin{proof} It is noted earlier that $\zeta$ is an eigenfunction of (\ref{FBP1}) with eigenvalue 0. By general spectral theory, since $\zeta$ is positive inside $\Sigma$ it must be a first eigenfunction and the corresponding eigenspace has dimension 1. The result then follows. 
%
\end{proof}
\begin{remark}
If $\Sigma$ is star-shaped, or equivalently a polar graph, then the assumption on $\zeta$ is satisfied. 
\end{remark}
\begin{proposition}
\label{0partialconverse}
Let $\Sigma^k\subset \mathbb{B}^{k+1}$ be a properly immersed FBMS with $\delta_1(L_J)=0$ and $\zeta$ positive everywhere inside then it must be an equatorial hyperplane. 
\end{proposition}

\begin{proof}
If $\lambda_1=0$ then equality happens in each estimate in the proof of Prop. \ref{lesseq0}. Thus, each $\nu_i$ has a decomposition 
\[\nu_i=w_i+v_i,\]
where $w_i\in \mathcal{J}_0^0$ and $v_i\in \hat{E}_0$. Since $\zeta$ is positive everywhere inside, by Lemma \ref{fixedzero}, $\mathcal{J}_0^0=\text{span}\{\zeta\}$. Consequently, $w_i=c_i\zeta$. Thus, we have the following equation along the boundary,
\begin{align*}
-\h(X,X)X_i&=D_\eta \nu_i\\
&=D_\eta(c_i\zeta+v_i)=c_i \h(X,X).
\end{align*}
If $|\h|$ is non-zero except for a set of measure zero along the boundary then each $X_i$ is a constant almost everywhere, a contradiction. Therefore, $|\h|$ vanishes on a set of positive measure. As a FBMS is real analytic, $|\h|$ must vanish everywhere and so $\Sigma$ must be an equatorial hyperplane.  
\end{proof}
\begin{corollary} Suppose $\Sigma^2\subset \mathbb{B}^{3}$ be a properly immersed FBMS of genus 0 and it has first Steklov eigenvalue 1 and first $J$-Steklov eigenvalue 0. Then $\Sigma$ must be an equatorial hyperplane.
\end{corollary}
\begin{proof}
By \cite[Prop 8.1]{FS16}, if $\Sigma$ is of genus 0 and has first Steklov 1, then $\Sigma$ must be star-shaped. Hence, the statement follows from Prop. \ref{0partialconverse}. 
\end{proof}
\section{Index Theorem}
\label{indextheorem}
In this section, we 
relate the index and nullity of the free boundary problem (using the general setup as in Definition \ref{generalsetup})to the index of the fixed boundary problem (\ref{FBP1}) and dimensions of $J$-Steklov eigenspaces. Unless otherwise stated, orthogonal decomposition is with respect to $S(.,.)$.  

 Recall Definitions \ref{bdrydata}, \ref{defineHpm} and observe the following.  

\begin{lemma}
\label{perpfun}
We have the following orthogonal decomposition with respect to the bilinear form $S(.,.)$,
\begin{align*} 
C^\infty(\Sigma) &=\bigoplus_{\delta< c} \hat{E}_\delta \oplus \bigoplus_{\delta\geq c} \hat{E}_\delta \oplus \overline{D_\eta \mathcal{J}_0^0};\\
\overline{D_\eta \mathcal{J}_0^0}&= \mathcal{J}_0^-\oplus (\mathcal{J}_0^-)^\perp.
\end{align*}
\end{lemma}
\begin{proof} 

For any function $u\in C^\infty(\Sigma)$, $u_{\mid \partial \Sigma}=b+h$ for $b\in D_\eta \mathcal{J}_0^0$ and $h\perp D_\eta \mathcal{J}_0^0$ with respect to $L^2(\partial \Sigma, \phi da)$. By Lemma \ref{existenceJ} and its following discussion, $h$ has an extension $\hat{h}$ such that $D_\eta \hat{h}\in (D_\eta \mathcal{J}_0^0)^\perp$. 
Then, for $w=u-\hat{h}$, $w\in \overline{D_\eta \mathcal{J}_0^0}$, $w_{\mid \partial \Sigma}=b$, and
\begin{align*}
S(w, \hat{h})&= \int_{\Sigma}-w J\hat{h}+\int_{\partial \Sigma}(D_\eta \hat{h}-\alpha \hat{h})\phi b,\\
&=\int_{\partial \Sigma}b \phi D_\eta \hat{h},\\
&=0.
\end{align*}
The last equality follows because $D_\eta \hat{h}\perp D_\eta \mathcal{J}_0^0$. 

\end{proof}
\begin{remark}
$(\mathcal{J}_0^-)^\perp\subset \overline{D_\eta \mathcal{J}_0^0}$ is an infinite dimensional vector space which includes $\mathcal{J}_0^0$ as a subspace. 
\end{remark}
We observe the following properties of $(\mathcal{J}_0^-)^\perp\subset \overline{D_\eta \mathcal{J}_0^0}$. 
\begin{lemma} 
\label{PropertyKB}
For each $u\in (\mathcal{J}_0^-)^\perp$ let $W_u=\{f\in (\mathcal{J}_0^-)^\perp, f_{\mid \partial \Sigma}=k u_{\mid \partial \Sigma}\}$. 
\begin{itemize}
\item[a.] If $u_{\mid \partial \Sigma}\equiv 0$ then $S(\dd,\dd)$ is non-negative definite on $W_u$.
\item[b.] If $u_{\mid \partial \Sigma} \not \equiv 0$ then $S(\dd,\dd)$ restricted to $W_u$ has index exactly equal to 1. 
\end{itemize} 
\end{lemma}
\begin{proof} For $u\in (\mathcal{J}_0^-)^\perp$, there is a unique $w\in \mathcal{J}_0^0$ such that $D_\eta w=u_{\mid \partial \Sigma}$ and for any constant $c$, $u+cw\in W_u$. 

\textbf{a.} Since $u\perp \mathcal{J}_0^-$ with respect to $S(\dd,\dd)$, for each $h\in \mathcal{J}_0^-$ and $J h=-\lambda h$, 
\begin{align*}
0=S(u,h)&=\int_{\Sigma}-uJh+\int_{\partial \Sigma} \phi u(D_{\eta} h-\alpha h)\\
&=\int_\Sigma \lambda uh.
\end{align*}
The second equality follows from $u_{\mid \partial \Sigma}\equiv 0$. Since each $\lambda\neq 0$, $u\perp \mathcal{J}_0^-$ with respect to $L^2(\Sigma)$ and, since $u_{\mid \partial \Sigma}\equiv 0$, the statement follows from the min-max characterization of eigenvalues of (\ref{FBP1}).\\

\textbf{b.} We compute,
\begin{align*}
S(u+cw,u+cw)&= S(u,u)+c^2 S(w,w)+2cS(u,w),\\
&=S(u,u)+2c\int_{\partial \Sigma}\phi u(D_\eta w),\\
&=S(u,u)+2c\int_{\partial \Sigma}\phi u^2.
\end{align*}
Since it is possible to choose $c$ to make the expression negative, the index of $S(\dd,\dd)$ restricted to $W_u$ is at least one. \\

To show that the index is exactly equal to 1, let $W'\subset W_u$ be a maximal space in which $S(.,.)$ is negative definite. If $\dim(W')\geq 2$, then there exist linearly independent functions $h, f\in W'\subset W_u$ such that 
\[h_{\mid \partial \Sigma}=k f_{\mid \partial \Sigma}=D_\eta w.\]
Therefore,
\[(h-kf)_{\mid \partial \Sigma}=w-w={0}.\] 
By part (a), $S(h-kf,h-kf)\geq 0$. Since $0\not\equiv u-kv\in W'$, this gives a contradiction to the definition of $W'$. Thus, the statement follows. 
\end{proof}

Now we are ready to characterize the index of $S(\dd,\dd)$.

\begin{theorem}\label{structureWminus}
Let $(\Sigma, \partial \Sigma)$ be a smooth compact Riemannian manifold with boundary. Given $\phi, m\in C^\infty(\Sigma)$ such that $\phi>0$ and a constant $\alpha$, $S(\dd,\dd)$ is defined as in Definition \ref{generalsetup}. Then the index of $S(.,.)$ is equal to  
\[\dim(\mathcal{J}_0^0)+\dim(\mathcal{J}_0^-)+\dim(\bigoplus_{\delta<\alpha} E_{\delta}).\]
Also any maximal space on which $S(.,.)$ is negative definite, by projection, is isomorphic to $\bigoplus_{\delta<\alpha} \hat{E}_{\delta}\oplus \mathcal{J}_0^- \oplus U$ where $U$ is constructed in the proof of Prop. \ref{Morseatleast}.
\end{theorem}
The proof is divided into Prop. \ref{Morseatleast} and Prop. \ref{Morseatmost} below.
\begin{proposition}
\label{Morseatleast}
The index of $S(.,.)$ is at least \[\dim(\mathcal{J}_0^0)+\dim(\mathcal{J}_0^-)+\dim(\bigoplus_{\delta<\alpha} E_{\delta}).\]
\end{proposition}
\begin{proof} The idea is to construct a space of this dimension on which $S(\dd,\dd)$ is negative definite. Recall that, for any constant $c$ and $\X=(\mathcal{J}_0^-)^\perp\subset \overline{D_\eta \mathcal{J}_0^0}$, 
\[C^{\infty}(\Sigma)=\bigoplus_{\delta < c} \hat{E}_\delta \oplus \bigoplus_{\delta\geq c} \hat{E}_\delta \oplus \mathcal{J}_0^-\oplus \X,\] where the direct sum decomposition is orthogonal with respect to $S(\dd,\dd)$. 
 
From the construction (see Definitions \ref{fixed} and \ref{defineHpm} and also equation (\ref{QandS})), it is clear that $S(\dd,\dd)$ is negative definite on $\mathcal{J}_0^-$ and $\bigoplus_{\delta < \alpha} \hat{E}_\delta$.

By Lemma \ref{PropertyKB}, a function $u\in \X$ that is not identically zero gives rise to a one-dimensional subspace on which $S(.,.)$ is negative definite. Furthermore, the construction is solely dependent on the boundary value of $u$ and independent of $\mathcal{J}_0^-$ and $\bigoplus_{\delta < \alpha} \hat{E}_\delta$. As a result, we expect there to be $\dim(\mathcal{J}_0^0)$ additional independent deformations on which $S(\dd,\dd)$ is negative definite. \\

Here is the precise construction. Let $B=D_\eta \mathcal{J}_0^0$ and $N=\dim(B)=\dim(\mathcal{J}_0^0)$. For each $b_i\in B$, see (\ref{bdrydata}), let $b_i'$ be an extension such that $b_i'\in \X=(\mathcal{J}_0^-)^\perp$. 
Since both the $\phi-$weighted $L^2$ inner product and the bilinear form $S(\dd,\dd)$ are symmetric, there is a set of functions denoted by the same notation $b'_i$, $i=1,... ,N$ such that,
\begin{itemize}
\item $b'_i \in X$.
\item $\int_{\partial \Sigma}\phi b'_i b'_j=\delta_{ij}$.
\item $S(b'_i,b'_j)=0$ if $i\neq j$.
\end{itemize} 

By Lemma \ref{PropertyKB}, for each $b'_i$ we can choose $u_i=b'_i+c_i w_i\in \X$, such that 
\begin{itemize}
\item $w_i\in\mathcal{J}_0^0$,
\item $u_i=b'_i=D_\eta w_i$ on $\partial \Sigma$,
\item $S(u_i,u_i)<0$.
\end{itemize}
\textbf{Claim:} $S(.,.)$ is negative definite on $U=\text{span}(u_1,...u_{N})$.\\
\textbf{Proof of the claim:} We compute, for $i\neq j$,
\begin{align*}
S(u_i,u_j)&=S(b'_i+c_i w_i,b'_j+c_jw_j)\\
&=S(b'_i,b'_j)+c_iS(w_i, b'_j)+c_jS(w_j,b'_i)+c_ic_jS(w_i,w_j)\\
&=S(b'_i,b'_j)+c_i\int_{\partial \Sigma}\phi (D_\eta w_i) b'_j+c_j\int_{\partial \Sigma}\phi (D_\eta w_j) b'_i\\
&=S(b'_i,b'_j)+(c_i+c_j) \int_{\partial \Sigma} \phi b'_i b'_j\\
&=0.
\end{align*}
Therefore, the claim follows. \\

Since $U\subset \X$, by Lemma \ref{perpfun}, $S(.,.)$ is negative definite on $W=U\oplus \mathcal{J}_0^-\oplus \bigoplus_{\delta < \alpha} \hat{E}_\delta$. Clearly, $\dim(U)=\dim(B)=\dim(\mathcal{J}_0^0)$ so the statement follows. 
\end{proof}
\begin{proposition}
\label{Morseatmost}
The index of $S(.,.)$ is at most \[\dim(\mathcal{J}_0^0)+\dim(\mathcal{J}_0^-)+\dim(\bigoplus_{\delta<\alpha} E_{\delta}).\]
\end{proposition}
\begin{proof}
Let $W=U\oplus \mathcal{J}_0^-\oplus \bigoplus_{\delta < \alpha} \hat{E}_\delta$ as in Prop. \ref{Morseatleast} and $W'$ be a maximal subspace of $C^\infty(\Sigma)$ on which $S(\dd,\dd)$ is negative definite. 

We consider the projection of $W'$ to $W$, $P=\text{Proj}^{S}_W(W')$, with respect to the bilinear symmetric form $S(.,.)$. \\

\textbf{Claim:} $P=W$. In other words, the projection is onto. 

\textbf{Proof of claim:} If not then there exists a non-trivial function $u$ such that,
\begin{itemize}
\item $u\in W$,
\item $S(u,w')=0$ for any $w'\in W'$. 
\end{itemize}
It follows immediately that $S(\dd,\dd)$ is negative definite on $W'\oplus \text{span}(u)$, a contradiction to the maximality of $W'$. Thus the claim follows. \\

Next, if $\dim(W')>\dim(W)$, then, by the dimension theorem, the kernel of the projection is nontrivial. If $\mu$ is in the kernel, then
\begin{itemize}
\item For any $w\in W$, $S(\mu,w)=0$,
\item $S(\mu,\mu)<0$. 
\end{itemize}
By Lemma \ref{perpfun}, since $\mu$ is $S$-orthogonal to $W=U\oplus \mathcal{J}_0^-\oplus \bigoplus_{\delta < \alpha} \hat{E}_\delta$, we can write $\mu=v+h$ such that:
\begin{itemize}
\item $v\in \X=(\mathcal{J}_0^-)^{\perp}\subset \overline{D_\eta \mathcal{J}_0^0}$ and $h\in \bigoplus_{\delta \geq \alpha} \hat{E}_\delta$,
\item $v$ is $S$-orthogonal to $U$. In other words, $S(u,v)=0$ for any $u\in U$.
\end{itemize}

\textbf{Claim:} $S(v,v)\geq 0$.

\textbf{Proof of the claim:}
Since $v\in \X\subset \overline{D_\eta \mathcal{J}_0^0}$, $v_{\mid \partial \Sigma}\in D_\eta \mathcal{J}_0^0$. If $v_{\mid \partial \Sigma}\equiv 0$, then the claim follows from Lemma \ref{PropertyKB}(a). Otherwise, there is $u\in U$ such that $u=v$ on $\partial \Sigma$. By the same argument as in Lemma \ref{PropertyKB},
\begin{align*}
0 &\leq S(u-v,u-v),\\
&=S(u,u)+S(v,v)-2S(u,v),\\
&=S(u,u)+S(v,v) < S(v,v).
\end{align*}
The last inequality follows because $S(u,u)<0$.
So the claim follows. \\

Furthermore, $S(h,h)\geq 0$ by the variational characterization for $\bigoplus_{\delta \geq \alpha} \hat{E}_\delta$ and (\ref{QandS}). By Lemma \ref{perpfun}, $S(h,v)=0$. Putting everything together, we have,
\[S(\mu,\mu)=S(h+v,h+v)=S(h,h)+S(v,v)\geq 0.\]

That contradicts the fact that $\mu$ is in a space on which $S(.,.)$ is negative definite. Therefore, $\dim(W')\leq \dim(W).$ 
\end{proof}

The same method as above also yields the computation of the nullity for $S(.,.)$. 
\begin{theorem}\label{nullityS} Let $(\Sigma, \partial \Sigma)$ be a smooth compact Riemannian manifold with boundary. Given $\phi, m\in C^\infty(\Sigma)$ such that $\phi>0$ and a constant $\alpha$, the nullity of $S(.,.)$ is equal to the dimension of $E_\alpha$ (the eigenspace of $L_J$ with eigenvalue $\alpha$). 
\end{theorem}

Now Theorem \ref{main1} follows as a consequence.
\begin{proof} 
For a FBMS $\Sigma^k\subset \Omega^{k+1}$ with $\h^{\partial \Omega}(\nu,\nu)=c$ let  
\begin{align*}
\alpha&=-c,\\
\phi &\equiv 1,\\
m &=(\text{Rc}(\nu,\nu)+|\h^\Sigma|^2).
\end{align*}
Then, the result follows from Theorems \ref{structureWminus} and \ref{nullityS}.  
\end{proof} 
\begin{remark}\label{equatorindex1} For an equatorial hyperplane in a Euclidean ball, the Jacobi operator becomes the Laplacian, and so we immediately recover the result that it has Morse index 1.
\end{remark}

When $\h^{\partial \Omega}(\nu,\nu)$ is not constant we let  
\begin{align*}
\alpha_I &=\inf_{\partial \Sigma}-\h^{\partial \Omega}(\nu,\nu),\\
\alpha_S &=\sup_{\partial \Sigma}-\h^{\partial \Omega}(\nu,\nu).
\end{align*}
If $\phi$ is the harmonic extension of $-\h^{\partial \Omega}(\nu,\nu)$ then, immediately, $\alpha_I\leq \phi\leq \alpha_S$. 
Now, for $x\in \{I, S\}$, we define, 
\begin{align*}
m_x &=\alpha_x(\text{Rc}(\nu,\nu)+|\h|^2);\\
Q_x(h,f) &=\int_\Sigma \phi \nabla f\nabla h-m_x fh,\\
S_x(h,f) &=\int_\Sigma \phi \nabla f\nabla h-m_x fh-\alpha_x\int_{\partial \Sigma}\phi fh \\
&=Q_x(h,f)-\alpha_x\int_{\partial \Sigma}\phi fh.
\end{align*} 
We observe, for $\mathcal{S}$ the index form, 
\begin{align*}
\alpha_I \mathcal{S}(u,u)&= \int_\Sigma \alpha_I|\nabla u|^2-m_I u^2+\alpha_I\int_{\partial \Sigma}\h^{\partial \Omega}(\nu,\nu) u^2\\
&\leq S_I(u,u),
\end{align*}
and
\begin{align*}
\alpha_S\mathcal{S}(u,u)&= \int_\Sigma \alpha_S|\nabla u|^2-m_S u^2+\alpha_S\int_{\partial \Sigma}\h^{\partial \Omega}(\nu,\nu) u^2,\\
&\geq S_S(u,u),
\end{align*}
So we obtain the following.
\begin{corollary} Let $\Omega^{k+1}$ be a manifold with convex boundary. Let $\Sigma^k\subset \Omega^{k+1}$ be a FBMS such that $0<\alpha_I\leq -\h^{\partial \Omega}(\nu,\nu)\leq \alpha_S$. Then the Morse index of $\Sigma$ is bounded between the indices of $S_I(.,.)$ and $S_{S}(.,.)$ defined above. 
\end{corollary}
\begin{remark} It was pointed by an anonymous referee that it is possible to obtain estimates by choosing $\phi=1$, $\alpha$ equal to either $\alpha_I, \alpha_S$ or minimum and maximum of principal curvatures of $\partial \Sigma$. 
\end{remark}
\subsection{FBMS with Index 4}
\label{index4}
By Theorem \ref{structureWminus} and \ref{lowH1}, the index of any non-equatorial submanifold in $\mathbb{B}^{k+1}$ is at least $k+2$. In this section, we consider the critical case of a FBMS in $\mathbb{B}^3$ with index 4 (it is shown in Section \ref{fbma} that there exists such a FBMS). This is analogous to the critical case of a minimal submanifold in $\mathbb{S}^3$ with index 5.  

The following was known to R. Schoen and A. Fraser and a proof is provided for completeness. 
\begin{theorem}
\label{index4Steklov1}
Suppose $\Sigma^2\subset \mathbb{B}^{3}$ be a properly immersed FBMS with Morse index 4. Then the first Steklov eigenvalue is 1. 
\end{theorem}
\begin{proof}
First, from Remark \ref{equatorindex1}, it follows that  $\Sigma$ can not be an equator. 

Suppose the first Steklov eigenvalue $\xi<1$ and let $u$ be an eigenfunction associated with $\delta$. That is,
\begin{align*}
\Delta u &=0 \text{~~on~~~~} \Sigma,\\
D_\eta u &=\xi u \text{~~on~~~~} \partial \Sigma.
\end{align*}
Let $\{e_i\}_{i=1}^3$ be an orthonormal basis of $\mathbb{R}^3$ and $V=\text{span}(u, X_1, X_2, X_3, \underline{1})$, where $\underline{1}$ is the function constantly equal to 1. Then each $X_i$ is a Steklov eigenfunction with eigenvalue 1 and $\underline{1}$ is an eigenfunction with eigenvalue 0.   
Also, because $\Sigma$ is not equatorial, $\dim(V)=5$.

Next we'll show that $S(\dd,\dd)$ is negative definite on $V$. For any $0\not\equiv v\in V$, there exist a constant vector $a\in \mathbb{R}^3$ and numbers $c, d$ ($|a|^2+c^2+d^2>0$) such that,
\[v=X_a+cu+d.\]
We compute,
\begin{align*}
S(v,v)&= \int_\Sigma \nabla v\nabla v-|A|^2 v^2-\int_{\partial \Sigma} v^2\\
&= \int_\Sigma -v(\Delta+|A|^2)v+\int_{\partial \Sigma} (D_\eta v-v)v\\
&=\int_\Sigma -|A|^2 v^2+\int_{\partial \Sigma} (X_a+c\xi u-X_a-cu-d)v\\
&=\int_\Sigma -|A|^2 v^2+\int_{\partial \Sigma} \Big((\xi-1)cu-d\Big)(X_a+cu+d)\\
&=\int_\Sigma -|A|^2 v^2+\int_{\partial \Sigma}(\xi-1)c^2u^2-d^2<0.
\end{align*} 
Thus, $S(.,.)$ is negative-definite on $V$ and the Morse index of $\Sigma$ is at least 5, a contradiction. Therefore, $\xi=1$.  
\end{proof}
\begin{proposition}
\label{4implyPos}
Suppose $\Sigma^2\subset \mathbb{B}^{3}$ be a properly immersed FBMS with Morse index 4. Then $\zeta=\left\langle{X,\nu}\right\rangle$ is positive everywhere inside. In particular, $\Sigma$ is stable with respect to variations fixing the boundary. 
\end{proposition}
\begin{proof}
As observed in Subsection \ref{FBMSinball}, $\zeta$ is an eigenfunction for the fixed boundary problem (\ref{FBP1}). If it is not positive everywhere inside, then it must not be the first eigenfunction. Consequently, 
\[\dim(\mathcal{J}_0^-)+\dim(\mathcal{J}_0^0)>1.\] 
Now if $\Sigma$ is not equatorial, then,  by Lemma \ref{lowH1}, $\dim(\bigoplus_{\delta<1}E_\delta)\geq 3$. Hence, by Theorem \ref{structureWminus}, the Morse index is bigger than 4, a contradiction. 
\end{proof}
Immediate consequences are the following. 
\begin{corollary}
Suppose $\Sigma^2\subset \mathbb{B}^{3}$ be an embedded FBMS with Morse index 4. Then $\Sigma$ must be star-shaped. In particular, $\Sigma$ has genus 0.
\end{corollary}
\begin{proof}
By Prop. \ref{4implyPos}, $\zeta=\left\langle{X,\nu}\right\rangle$ must be positive everywhere inside. As $\Sigma$ is embedded, it must be star-shaped. The result then follows. 
\end{proof}
\begin{corollary}
\label{2boundary}
Suppose $\Sigma^2\subset \mathbb{B}^{3}$ be an embedded FBMS with Morse index 4 and two boundary components. Then $\Sigma$ must be congruent to the critical catenoid.
\end{corollary}
\begin{proof}
By the previous corollary, $\Sigma$ is star-shaped and has genus zero. Such a surface with two boundary components must have the topology of an annulus. Furthermore, by Prop \ref{index4Steklov1}, if $\Sigma$ has index 4, then its first Steklov eigenvalue is 1. Finally \cite[Theorem 6.6]{FS16} implies that the surface must be congruent to the critical catenoid. 
\end{proof}

\section{Index of the Critical Catenoid}
\label{fbma}
In this section, we study Jacobi fields of the critical catenoid and prove Theorem \ref{main2}. The critical catenoid is the only known example of a free boundary minimal annulus (FBMA) in $\mathbb{B}^3$. That is, it is a FBMS in $\mathbb{B}^3$ of genus zero having two boundary components. Such a surface is described by a conformal harmonic map $X:M \mapsto \mathbb{B}^3$, where $M$ the cylinder $[-T, T]\times S^1$ with coordinates $(t,\theta)$. 

For constants $c, T$ to be determined later, the critical catenoid is congruent to the immersion
\begin{equation}
\label{cricat}
X(t,\theta) =c(\cosh{t}\cos{\theta}, \cosh{t}\sin{\theta}, t).
\end{equation}
Consequently, its tangent vector fields are
\begin{align*}
X_t&=c(\sinh{t}\cos{\theta}, \sinh{t}\sin{\theta}, 1),\\
X_\theta &= c(-\cosh{t}\sin{\theta}, \cosh{t}\cos{\theta}, 0).
\end{align*}
To satisfy the free boundary conditions, $T$ and $c$ are determined by
\begin{align*}
\cosh{T}&=T\sinh{T},\\
c &= \frac{1}{\sqrt{T^2+\cosh^2{T}}},\\
 &=\frac{1}{T\cosh{T}}.
\end{align*}
That is, \[T\approx 1.2, ~~~\cosh{T}\approx 1.81,~~~\sinh{T}\approx 1.51, ~~~\tanh{T}\approx .83.\]
Next, we compute its unit normal and second fundamental form:
\begin{align*}
\nu &= -\frac{X_t\times X_\theta}{|X_t\times X_\theta|}=\frac{1}{\cosh{t}}(\cos{\theta}, \sin{\theta}, -\sinh{t}),\\
\left\langle{X, \nu}\right\rangle &=c(1-\frac{t\sinh{t}}{\cosh{t}}),\\
X_{tt} &= c(\cosh{t}\cos{\theta}, \cosh{t}\sin{\theta}, 0),\\
X_{\theta \theta} &=c(-\cosh{t}\cos{\theta}, -\cosh{t}\sin{\theta}, 0),\\
\h_{tt} =-\h_{\theta\theta} &=c,\\
|X_t|=|X_\theta|&= c\cosh{t},\\
|\h|^2&=\frac{2}{c^2 \cosh^4{t}}.
\end{align*}
The normal derivative along the boundary is given by
\[D_\eta =\frac{1}{c\cosh{T}}D_{\pm t}=T D_{\pm t}.
\]  
It follows that 
\begin{align*}
D_{\eta}(X_1)&= \frac{1}{|X_t|}(D_{\pm t} X_1)= \frac{1}{c\cosh{T}}c\sinh{T}\cos{\theta}\\
&=c\cosh{T}\cos{\theta}=X_1.
\end{align*}
Similarly, 
\[D_{\eta}(X_2)=X_2, ~~D_{\eta}(X_3)=X_3.\]
So the coordinate functions are eigenfunctions with Steklov eigenvalue 1, as stated in the introduction. In fact, it was shown that they are first eigenfunctions \cite{FS11}. For Jacobi-Steklov eigenvalues, we compute: 
\begin{align*}
D_{\eta}(\left\langle{X,\nu}\right\rangle)&= \frac{1}{|X_t|}(D_{\pm t} c(1-\frac{t\sinh{t}}{\cosh{t}}))\\
&= \frac{-1}{\cosh{T}}\frac{T+\sinh{T}\cosh{T}}{\cosh^2{T}}=-\frac{{T}}{\cosh{T}}\\
&=-\frac{|\h|}{\sqrt{2}},\\
D_{\eta}(\nu_1)&= \frac{1}{|X_t|}(D_{\pm t} \frac{\cos{\theta}}{\cosh{t}})\\
&= \frac{-\cos{\theta}}{c\cosh{T}}\frac{\sinh{T}}{\cosh^2{T}}=\frac{-1}{\cosh{T}}\cos{\theta}\\
&=-\frac{|\h|}{\sqrt{2}}X_1=-\nu_1,\\
D_{\eta}(\nu_2)&=\frac{1}{|X_t|}(D_{\pm t} \frac{\sin{\theta}}{\cosh{t}})=-\nu_2,\\
D_{\eta}(-\nu_3)&= \frac{1}{|X_t|}(D_{\pm t} \frac{\sinh{t}}{\cosh{t}})\\
&= \frac{1}{c\cosh{T}}\frac{\pm 1}{\cosh^2{T}}\\
&=\frac{|\h|}{\sqrt{2}}X_3=\frac{1}{\sinh^2{T}}(-\nu_3).
\end{align*}
So, in this case, the components of the normal vector are eigenfunctions for $L_J$. Indeed, it is possible to compute all eigenvalues and eigenfunctions in this setting. \\

The Jacobi operator is 
\begin{align*}
J=\Delta_\Sigma+|\h|^2 &=\frac{1}{c^2 \cosh^2{t}}(\Delta+\frac{2}{\cosh^2{t}}).
\end{align*}
Eigenvalue $\delta$ of $L_J$ is associated with the following PDE on $M=[-T,T]\times S^1$,
\[
\begin{cases}
(\Delta+\frac{2}{\cosh^2{t}})u = 0 &\text{on $\Sigma$},\\
 TD_{\pm t} u = \delta u ~ ~ \text{where}~~ \{t=\pm T\} & \text{on $\partial \Sigma$}.
\end{cases}
\]

As $M$ is rotationally symmetric, the method of separation of variables is applicable. Let $u=f(t)g(\theta)$. Then the PDE becomes 
\[
\begin{cases}
\frac{f''}{f}+\frac{g''}{g}&=-\frac{2}{\cosh^2{t}},\\
Tf' &=\delta f ~~~\text{where}~~ \{t=T\},\\
-Tf' &=\delta f ~~~\text{where}~~ \{t=-T\}.
\end{cases}
\]
Since $g$ is a function of $\theta$, $\frac{g''}{g}=-d$. That is, $g$ is an eigenfunction of the Laplacian on $\mathbb{S}^1$. Consequently, $d=n^2$, for some non-negative integer $n$ and
\begin{align*}
g &=c_1 \cos(n\theta)+c_2 \sin(n\theta), ~~\text{for $n>0$},\\
g &=c_1 , ~~\text{for $n=0$}.
\end{align*}
Thus, the PDE is further reduced to
\begin{equation}
\label{pdeannulus}
\begin{cases}
(\partial_t^2+\frac{2}{\cosh^2{t}}-n^2)f &=0.\\
Tf' &=\delta f ~~~\text{where}~~ \{t=T\}, \\
-Tf' &=\delta f ~~~\text{where}~~ \{t=-T\}.
\end{cases}
\end{equation}
Thus, it is important to understand the operator $L_{k}:=\partial_t^2+\frac{2}{\cosh^2{t}}-k $. The analysis of this operator is well known and we follow the treatment in \cite{Wiygul15}. First, we define
\begin{align*}
D^+ &:=\partial_t+\tanh{t},\\
D^- &:=\partial_t-\tanh{t}.
\end{align*} 
Then,
\begin{align*}
D^+ D^- &=\partial_t^2-1 :=A_0-1,\\
D^- D^+ &=\partial_t^2+\frac{2}{\cosh^2{t}}-1=L_0-1.
\end{align*}
Therefore, for $A_k = A_0-k,$
\begin{align*}
L_k(D^{-}u)&=D^-(A_k(u)),\\
A_k(D^+u)&= D^{+}(A_k(u)).
\end{align*}
The following result is immediate.
\begin{lemma} \label{D-andD+}We have the following:
\begin{itemize}
\item[a.] If $u$ is in the kernel of $L_k$ then $D^{+}u$ is in the kernel of $A_k$.
\item[b.] If $u$ is in the kernel of $A_k$ then $D^{-}u$ is in the kernel of $L_k$.
\item[c.] If $k\neq 1$ then $\text{Ker}(L_k)=D^{-}(\text{Ker}(A_k))$.
\item[d.] If $k=1$ then $D^+\text{Ker}(L_1)=\text{Ker}(D^-).$
\end{itemize}
 \end{lemma}
 \begin{proof}
 (a) and (b) are obvious from the computation above. 
 
 For (c), one direction follows immediately from (b). For the reverse, let $u\in \text{Ker}(L_k)$. Then, by part(a),
 \[ D^{+}u=v \in \text{Ker}(A_k).\]
 Therefore, 
 \begin{align*}
  D^{-}(v)&=D^{-}D^{+}(u)=(L_k+k-1)(u)=(k-1)u.
   \end{align*}
Since $k\neq 1$, $u\in D^{-}\text{Ker}(A_k)$. 
 
For part (d), the equation above implies that $D^-(v)=0$. Thus, $D^+\text{Ker}(L_1)\subset \text{Ker}(D^-)$. Since, $\dim\text{Ker}(D^-)=1=D^+\text{Ker}(L_1)$ the result follows.   
 \end{proof}

The kernel of $A_k=\partial_t^2-k$ is standard so Lemma \ref{D-andD+} gives solutions of (\ref{pdeannulus}). 
\begin{theorem}
\label{alleigencat}
For the critical catenoid given as in (\ref{cricat}), eigenfunctions and eigenvalues of $L_J$ are given by solutions of (\ref{pdeannulus}) for each non-negative integer n. In particular, they are listed below:
\begin{itemize}
\item For $n=0$: 
\begin{align*}
u &= \tanh{t},\\
\delta &=\frac{1}{\sinh^2{T}}<1.
\end{align*}
\item For $n=1$, there are two cases: 
\begin{align*}
u &= (c_1\cos\theta+c_2\sin{\theta})\frac{1}{\cosh{t}},\\
\delta &=-1,
\end{align*}
or
\begin{align*}
u &= (c_1\cos\theta+c_2\sin{\theta})(\sinh{t}+\frac{t}{\cosh{t}}),\\
\delta &=1. 
\end{align*}
\item For each $n\geq 2$, there are two cases:
\begin{align*}
u&=(c_1\cos (n\theta)+c_2\sin(n\theta))\Big((n-\tanh{t})e^{nt}+(n+\tanh{t})e^{-nt}\Big),\\
\delta &= T\frac{\Big(n(n-\tanh{T})-\frac{1}{\cosh^2{T}}\Big) e^{nT}-\Big(n(n+\tanh{T})-\frac{1}{\cosh^2{T}}\Big) e^{-nT}}{(n-\tanh{T})e^{nT}+(n+\tanh{T})e^{-nT}},
\end{align*}
or
\begin{align*}
u&=(c_1\cos (n\theta)+c_2\sin(n\theta))\Big((n-\tanh{t})e^{nt}-(n+\tanh{t})e^{-nt}\Big),\\
\delta &= T\frac{\Big(n(n-\tanh{T})-\frac{1}{\cosh^2{T}}\Big) e^{nT}+\Big(n(n+\tanh{T})-\frac{1}{\cosh^2{T}}\Big) e^{-nT}}{(n-\tanh{T})e^{nT}-(n+\tanh{T})e^{-nT}}.
\end{align*}
\end{itemize}
\end{theorem}
\begin{proof}
When $n=0$, the solution to (\ref{pdeannulus}) is given by linear combinations of 
\begin{align*}
f^0_1 &= \tanh{t},\\
f^0_2 &= 1-t\tanh{t}.
\end{align*}
As $f_2^0$ is a multiple of $\zeta\in \mathcal{J}_0^0$ so an eigenfunction is only a multiple of $f^0_1$.
Thus, the eigenvalue is \[\lambda=\frac{1}{\sinh^2{T}}<1.\]
 
When $n=1$, we observe that $\text{Ker}(D^-)$ has dimension one. By the method of integrating factors, $\text{Ker}(L_1)$ is given by linear combinations of 
\begin{align*}
f^1_1 &= \frac{1}{\cosh{t}},\\
f^1_2 &= \sinh{t}+\frac{t}{\cosh{t}}.
\end{align*}
If $f =a f_1^1+b f_2^1$, then
\begin{align*}
f' &= -\frac{a\sinh{t}}{\cosh^2{t}}+b(\cosh{t}+\frac{\cosh{t}-t\sinh{t}}{\cosh^2{t}}),\\
f'_{\mid \partial \Sigma} &= -\frac{a\sinh{t}}{\cosh^2{t}}+b\cosh{t},\\
Tf'_{\mid T}+Tf'_{\mid -T}=2bT\cosh{T} &=\delta (f_{\mid T}-f_{\mid -T})=2\delta b (\sinh{T}+\frac{T}{\cosh{T}}),\\
Tf'_{\mid T}-Tf'_{\mid -T}=-2\frac{a}{\cosh{T}} &=\delta (f_{\mid T}+f_{\mid -T})=2\delta \frac{a}{\cosh{T}}.
\end{align*}
There are 2 cases:
\begin{enumerate}
\item $a\neq 0$, $b=0$, $\delta=-1$.
\item $a=0$, $b\neq 0$, $\delta= \frac{T\cosh{T}}{\sinh{T}+\frac{T}{\cosh{T}}}=\frac{T^2\sinh^2{T}}{\sinh^2{T}+1}=1.$
\end{enumerate}

When $n\geq 2$, the solution to (\ref{pdeannulus}) is given by linear combinations of 
\begin{align*}
f^n_1 &= (n-\tanh{t})e^{nt},\\
f^n_2 &= (n+\tanh{t})e^{-nt}.
\end{align*}
If $f =a f_1^n+b f_2^n$, then
\begin{align*}
f' &= \Big(n(n-\tanh{t})-\frac{1}{\cosh^2{t}}\Big)a e^{nt}-\Big(n(n+\tanh{t})-\frac{1}{\cosh^2{t}}\Big)b e^{-nt},\\
aTf'_{\mid T}+bTf'_{\mid -T} &=(a^2-b^2)T e^{nT}\Big(n(n-\tanh{T})-\frac{1}{\cosh^2{T}}\Big)\\
 &=\delta (af_{\mid T}-bf_{\mid -T})=\delta (a^2-b^2)e^{nT}(n-\tanh{T}),\\
bTf'_{\mid T}+aTf'_{\mid -T} &=(a^2-b^2)T e^{-nT}\Big(n(n+\tanh{T})-\frac{1}{\cosh^2{t}}\Big)\\
 &=\delta (bu_{\mid T}-au_{\mid -T})=\delta (b^2-a^2)e^{-nT}(n+\tanh{T}).
\end{align*}
Thus, $a=\pm b$ and, as a consequence, $f(t)=\pm f(-t)$ and 
\begin{align*}
\delta &= \frac{Tf'(T)}{f(T)}\\
&=T\frac{\Big(n(n-\tanh{T})-\frac{1}{\cosh^2{T}}\Big) e^{nT}\mp\Big(n(n+\tanh{T})-\frac{1}{\cosh^2{T}}\Big) e^{-nT}}{(n-\tanh{T})e^{nT}\pm(n+\tanh{T})e^{-nT}}.
\end{align*}
\end{proof}

We are now ready to prove one direction of Theorem \ref{main2}.
\begin{theorem}
\label{indexcat}
The critical catenoid has Morse index 4 and nullity 2.
\end{theorem}
\begin{proof}
We use the index formula from Theorem \ref{structureWminus}. As the critical catenoid is a polar-graph, by Lemma \ref{fixedzero} and its following remark, $\dim(\mathcal{J}_0^-)+\dim(\mathcal{J}_0^0)=1$. It remains to count eigenvalues less than 1 of (\ref{pdeannulus}).  

By Theorem \ref{alleigencat}, for $n\leq 1$, the eigenvalues less than 1 are $\frac{1}{\sinh^2{T}}$ (of multiplicity 1) and $-1$ (of multiplicity $2$).

For $n\geq 2$, we consider whether $\delta-1\geq 0$ for
\[\delta =T\frac{\Big(n(n-\tanh{T})-\frac{1}{\cosh^2{T}}\Big) e^{nT}-\Big(n(n+\tanh{T})-\frac{1}{\cosh^2{T}}\Big) e^{-nT}}{(n-\tanh{T})e^{nT}+(n+\tanh{T})e^{-nT}},\]
since the other eigenvalue is even bigger. $\delta-1$ is positive if the following function is greater than 1,
\[ \varphi(n)=\Big(\frac{(Tn-1)(n-\tanh{T})-\frac{T}{\cosh^2(T)}}{(Tn+1)(n+\tanh{T})-\frac{T}{\cosh^2(T)}}\Big)\frac{e^{nT}}{e^{-nT}}.
\]
The first ratio can be rewritten as
\begin{align*}
\frac{(Tn-1)(n-\tanh{T})-\frac{T}{\cosh^2(T)}}{(Tn+1)(n+\tanh{T})-\frac{T}{\cosh^2(T)}} &= \frac{Tn^2-n(1+T\tanh{T})+\tanh{T}-\frac{T}{\cosh^2(T)}}{Tn^2+n(1+T\tanh{T})+\tanh{T}-\frac{T}{\cosh^2(T)}}\\
&=\frac{an^2-bn+c}{an^2+bn+c}.
\end{align*}
Recall, $T\sinh{T}=\cosh{T}$ and 
\[T\approx 1.19968, ~~~\cosh{T}\approx 1.81,~~~\sinh{T}\approx 1.51, ~~~\tanh{T}\approx .83.\]
Thus, 
\[ a\approx 1.2, ~~~b=2, ~~~c\approx .4674.\]
It follows that $g(n)=\frac{an^2-bn+c}{an^2+bn+c}>\frac{1}{k}$ for $n\geq 2$ if
\begin{align*}
(4a+c)(k-1)-4(k+1) &>0\\
\leftrightarrow a+\frac{c}{4}-1 &>\frac{2}{k-1}.
\end{align*}
In particular, it is true for $k=11$. 

On the other hand, for $n\geq 2$, $f(n)=g(n)\frac{e^{nT}}{e^{-nT}}>54 g(n)$. Thus, for $n\geq 2$, all eigenvalues for $n\geq 2$ from Theorem \ref{alleigencat} are bigger than 1. We conclude that the critical catenoid has Morse index 4.

For the nullity, by Theorem \ref{nullityS}, we count eigenfunctions with eigenvalue 1 of \ref{pdeannulus}. By Theorem \ref{alleigencat}, the eigenvalue $1$ has multiplicity $2$, so the result follows.  
\end{proof}
Now we are ready to finish the proof of Theorem \ref{main2}.
\begin{proof}
One direction follows from Thm. \ref{indexcat} white the other from Cor. \ref{2boundary}.
\end{proof}

Finally, recall that for $M\in \mathfrak{so}(n)$, if $\left\langle{MX,\nu}\right\rangle \not\equiv 0$, it is an eigenfunction of $L_J$ with eigenvalue 1. Furthermore, $\left\langle{MX,\nu}\right\rangle$ is trivial if and only if $\Sigma$ is rotationally symmetric. Thus, it is possible to characterize the critical catenoid by its nullity.
\begin{theorem}
Let $\Sigma^k\subset \mathbb{B}^{k+1}$ be a non-equatorial FBMS. Then it has (free boundary) nullity 2 if and only if $k=2$ and $\Sigma$ is a critical catenoid.
\end{theorem}


\def\cprime{$'$}
\bibliographystyle{plain}
\bibliography{bio}

\end{document}